\documentclass[a4paper]{amsart}
\numberwithin{equation}{section}
\newtheorem{theorem}{Theorem}[section]
\newtheorem{corollary}{Corollary}[section]

\newtheorem{lemma}{Lemma}[section]

\theoremstyle{remark}
\newtheorem{remark}{Remark}[section]
\usepackage{amsmath, hyperref}
\usepackage{color}
\usepackage{graphicx}
\usepackage{subfigure}

\title[On certain subclasses of the close-to-convex functions]{On certain subclasses of the close-to-convex functions related with the second-order differential subordination}


\keywords{Univalent; Positive real part; Circular domains; Close-to-convex; Marx-Strohh\"{a}cker problem; Section sum.\\
*Corresponding Author}

\begin{document}
\begin{abstract}
Let $\mathcal{A}$ be the family of analytic and normalized functions in the open unit disc $|z|<1$.
In this article we consider the following classes
\begin{equation*}
 \mathcal{R}(\alpha,\beta):=\left\{ f\in \mathcal{A}: {\rm Re}\left\{f'(z)+\frac{1+e^{i\alpha}}{2}zf''(z)\right\}>\beta,\, |z|<1\right\}
\end{equation*}
and
\begin{equation*}
  \mathcal{L}_\alpha(b):=\left\{f\in\mathcal{A}:\left|f'(z)
  +\frac{1+e^{i\alpha}}{2}zf''(z)-b\right|< b,\, |z|<1 \right\},
\end{equation*}
where $-\pi<\alpha\leq \pi$, $0\leq \beta<1$ and $b>1/2$.
We show that if $f\in \mathcal{R}(\alpha,\beta)$, then ${\rm Re}\{f'(z)\}$ and ${\rm Re}\{f(z)/z\}$ are greater than $\beta$, and if $f\in\mathcal{L}_\alpha(b)$, then $0<{\rm Re}\{f'(z)\}<2b$. Also, some another interesting properties of the class $\mathcal{L}_\alpha(b)$ are investigated. Finally, the radius of univalence of 2-th section sum of $f\in \mathcal{R}(\alpha,\beta)$ is obtained.
\end{abstract}

\author[H. Mahzoon and R. Kargar] {H. Mahzoon and R. Kargar$^*$}
\address{Department of Mathematics, Islamic Azad University, Firoozkouh
Branch, Firoozkouh, Iran}
\email {mahzoon$_{-}$hesam@yahoo.com {\it (H. Mahzoon)}}
\address{Young Researchers and Elite Club,
Ardabil Branch, Islamic Azad University, Ardabil, Iran}
       \email{rkargar1983@gmail.com {\it (R. Kargar)}}

\maketitle

\section{Introduction}
Let $\Delta:=\{z\in\mathbb{C}:|z|<1\}$ where $\mathbb{C}$ is the complex plane. We denote by $\mathcal{B}$ the class of all analytic functions $w(z)$ in $\Delta$ with $w(0)=0$ and $|w(z)|<1$, and denote by $\mathcal{A}$ the class of all functions that are analytic and normalized in $\Delta$. 
The subclass of $\mathcal{A}$ consisting of univalent functions in $\Delta$ is denoted by $\mathcal{S}$. For functions $f$ and $g$ belonging to the class $\mathcal{A}$, we say that $f$ is subordinate to $g$ in the unit disk $\Delta$, written $f(z) \prec g(z)$ or $f\prec g$, if and only if there exists a function $w\in\mathcal{B}$ such that $f(z)=g(w(z))$ for all $z\in\Delta$.
In particular, if $g$ is univalent function in $\Delta$, then we have the following relation
\begin{equation*}
    f(z)\prec g(z) \Leftrightarrow f(0)=g(0)\quad {\rm and}\quad f (\Delta)\subset g(\Delta).
\end{equation*}
Denote by $\mathcal{S}^*$ and $\mathcal{K}$ the set of all starlike and convex functions in $\Delta$, respectively.
A function $f\in \mathcal{A}$ is said to be close-to-convex, if there exists
a convex function $g$ and $\delta\in\mathbb{R}$ such that
\begin{equation*}
  {\rm Re}\left\{e^{i\delta}\frac{f'(z)}{g'(z)}\right\}>0\quad(z\in\Delta).
\end{equation*}
The functions class which satisfy the last condition was introduced by Kaplan in \cite{Kaplan} and we denote by $\mathcal{CK}$. It is clear that if we take $g(z)\equiv z$ in the class $\mathcal{CK}$, then we have the Noshiro-Warschawski class as follows
\begin{equation*}
  \mathcal{C}:=\left\{f\in \mathcal{A}:\exists \delta\in\mathbb{R}; {\rm Re}\left\{e^{i\delta}f'(z)\right\}>0, \, z\in\Delta\right\}.
\end{equation*}
By the basic Noshiro-Warschawski lemma \cite[\S 2.6]{Duren}, we have $\mathcal{C}\subset\mathcal{S}$.

Here, we recall from \cite{sils}, two certain subclasses of analytic functions as follows
\begin{equation*}\label{ReL}
  \mathcal{L}_\alpha:=\left\{f\in\mathcal{A}:{\rm Re}\left\{f'(z)
  +\frac{1+e^{i\alpha}}{2}zf''(z)\right\}>0,\,  z\in\Delta \right\}
\end{equation*}
and
\begin{equation*}\label{Lb}
  \mathcal{L}_\alpha(b):=\left\{f\in\mathcal{A}:\left|f'(z)
  +\frac{1+e^{i\alpha}}{2}zf''(z)-b\right|< b,\, z\in\Delta \right\},
\end{equation*}
where $\alpha\in(-\pi,\pi]$ and $b>1/2$. Notice that if $b\rightarrow \infty$, then $\mathcal{L}_\alpha(b)\rightarrow \mathcal{L}_\alpha$. Also, $\mathcal{L}_\pi$ contains $\mathcal{L}_\alpha$ for each $\alpha$. On the other hand, Trojnar-Spelina \cite{spelina} showed that $\mathcal{L}_\alpha(b)\subset \mathcal{L}_\pi$, for every $\alpha\in(-\pi,\pi]$ and $b\geq1$.

By definition of subordination and this fact that the image of the function
\begin{equation}\label{phi}
\phi_b(z)=\frac{1+z}{1+\left(\frac{1}{b}-1\right)z}\quad(z\in \Delta,\, b>1/2),
\end{equation}
 is $\{w\in \mathbb{C}:|w-b|< b\}$ (see Figure \ref{Fig:1} for $b=3/2$), we have the following lemma.
 \begin{lemma}\label{lemma spelina}
{\rm(}see \cite{spelina}{\rm )}
A necessary and sufficient condition for $f$ to be in the class $\mathcal{L}_\alpha(b)$ is
\begin{equation*}
  f'(z)+\frac{1+e^{i\alpha}}{2}zf''(z)\prec \phi_b(z)\quad(z\in \Delta),
\end{equation*}
where $\phi_b$ is given by \eqref{phi}.
\end{lemma}
\begin{figure}[htp]
 \centering
 \includegraphics[width=8cm]{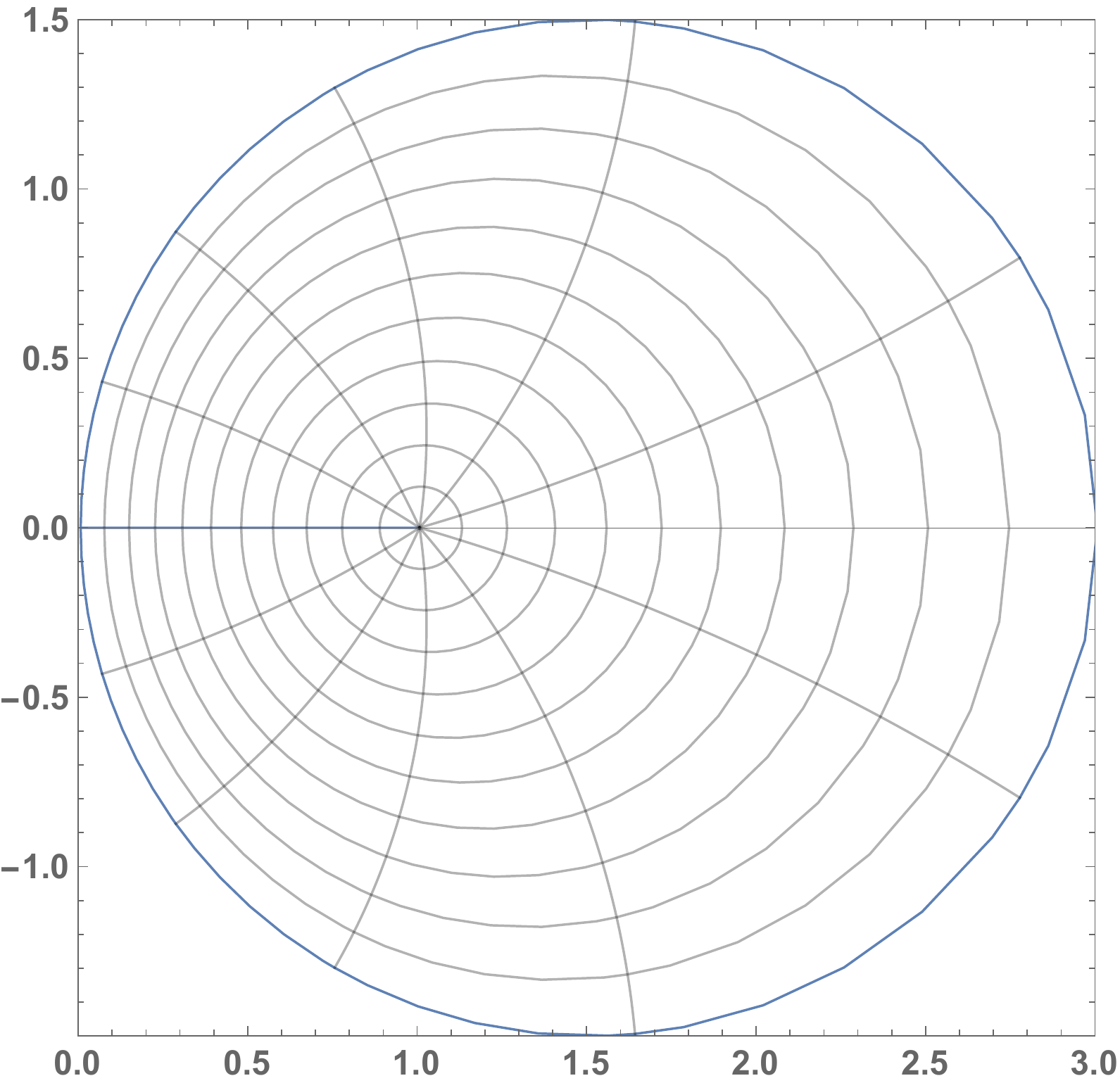}\\
 \caption{The boundary curve of $\phi_{3/2}(\Delta)$}
 \label{Fig:1}
\end{figure}

It is necessary to point out that the class $\mathcal{R}_\alpha(\varphi)$ including of all normalized analytic functions in $\Delta$ satisfying the following differential subordination
\begin{equation*}
  f'(z)+\frac{1+e^{i\alpha}}{2}zf''(z)\prec \varphi(z)\quad(z\in \Delta),
\end{equation*}
was studied extensively by Srivastava $et$ $al.$ (see \cite{sridorzap}), where the function $\varphi$ is analytic in the open unit disc $\Delta$ such that $\varphi(0)=1$.
Also, Chichra \cite{Chichra} studied the class of all functions whose derivative has positive real part in the unit disc $\Delta$. Indeed, he denoted by $\mathcal{F}_\gamma$ the class of functions $f\in\mathcal{A}$ which satisfying the following inequality
\begin{equation*}
  {\rm Re}\left\{f'(z)+\gamma zf''(z)\right\}>0\quad(z\in\Delta),
\end{equation*}
where $\gamma\geq0$, and showed that $\mathcal{F}_\gamma\subset \mathcal{S}$. Also, he proved that if $f\in\mathcal{F}_\gamma $ and ${\rm Re}\{\gamma\}\geq0$, then ${\rm Re}\{f'(z)\}>0$ in $\Delta$. Recent result, also was obtained by Lewandowski {\it et al.} in \cite{lewa et al}.

On the other hand, Gao and Zhou \cite{gaozhou} considered the class $R(\beta,\gamma)$ as follows:
\begin{equation*}
  R(\beta,\gamma)=\left\{f\in\mathcal{A}:{\rm Re}\left\{f'(z)+\gamma zf''(z)\right\}>\beta,\ \ \gamma>0,\, \beta<1,\, z\in \Delta\right\}.
\end{equation*}
They found the extreme points of $ R(\beta,\gamma)$, some sharp bounds of certain
linear problems, the sharp bounds for ${\rm Re}\{f'(z)\}$ and ${\rm Re}\{f(z)/z\}$ and determined the number
$\beta(\gamma)$ such that $R(\beta,\gamma)\subset\mathcal{S}^*$, where $\gamma$ is certain fixed number in $[1,\infty)$. Also, the class $R(\beta,\gamma)$ was studied by Ponnusamy and Singh when ${\rm Re}\{\gamma\}>0$, see \cite{PonSin}.

Motivated by the above classes, we define the class of all functions $f\in\mathcal{A}$, denoted by $\mathcal{R}(\alpha,\beta)$ which satisfy the condition
\begin{equation*}
  {\rm Re}\left\{f'(z)+\frac{1+e^{i\alpha}}{2}zf''(z)\right\}>\beta\quad(z\in\Delta),
\end{equation*}
where $0\leq \beta<1$ and $-\pi<\alpha\leq \pi$. It is obvious that $\mathcal{R}(\pi,\beta)$ becomes the class $\mathcal{C}(\beta)$, where
\begin{equation*}
  \mathcal{C}(\beta):=\left\{f\in \mathcal{A}: {\rm Re} \{f'(z)\}>\beta, z\in \Delta, 0\leq \beta<1\right\}.
\end{equation*}
The class $\mathcal{C}(\beta)$ was considered in \cite{goodman} and $\mathcal{C}(\beta)\subset \mathcal{S}$ when $0\leq \beta<1$. It follows from \cite[Theorem 5]{Chichra} that $\mathcal{R}(\alpha,0)\subset \mathcal{R}(\pi,0)\equiv \mathcal{C}(0)\equiv \mathcal{C}$. The class $\mathcal{R}(0,0)$ studied by Singh and Singh \cite{singh2}, and they showed that $\mathcal{R}(0,0)\subset \mathcal{S}^*$ \cite{singh2s}. Also, they found for $f\in \mathcal{R}(0,0)$ and $z\in \Delta$ that ${\rm Re} \{f(z)/z\}>1/2$ and $\mathcal{R}(0,\beta)\subset \mathcal{S}^*$ for $\beta\geq -1/4$. Silverman in \cite{sil94} improved this lower bound. He showed that $\mathcal{R}(0,\beta)\subset \mathcal{S}^*$ for $\beta \geq -0.2738$ and also found the smallest $\beta$ ($\beta\geq -0.63$) for which $\mathcal{R}(0,\beta)\subset \mathcal{S}$.

Since the function $z\mapsto (1+(1-2\beta)z)/(1-z)$ $(z\in\Delta, 0\leq \beta<1)$ is univalent and maps $\Delta$ onto the right half plane, having real part greater than $\beta$, we have the following lemma directly. With the proof easy, the details are omitted.
\begin{lemma}
  A function $f\in\mathcal{A}$ belongs to the class $\mathcal{R}(\alpha,\beta)$ if, and only if,
  \begin{equation*}
    f'(z)+\frac{1+e^{i\alpha}}{2}zf''(z)\prec \frac{1+(1-2\beta)z}{1-z}\quad(z\in\Delta, 0\leq \beta<1, -\pi<\alpha\leq \pi).
  \end{equation*}
\end{lemma}

To prove of our main results we need the following lemma.

\begin{lemma}\label{MIMO}
\cite[p. 35]{MM-book} Let $ \Xi$ be a simply connected domain in the complex plane
$\mathbb{C}$ and let $t$ be a complex number such that
${\rm Re}\{t\} > 0$. Suppose that a function $
\psi:\mathbb{C}^2\times \Delta\rightarrow \mathbb{C}$ satisfies the
condition
\begin{equation*}
   \psi(i\rho,\sigma;z)\not\in
   \Xi
\end{equation*}
for all real $\rho,\sigma\leq-\mid
t-i\rho\mid^2/(2{\rm Re}\,t)$ and all $z\in\Delta$. If the
function $p(z)$ defined by $p(z) =t+t_1z+t_2z^2+\cdots$ is analytic
in $\Delta$ and if
\begin{equation*}
   \psi(p(z),zp'(z);z)\in \Xi,
\end{equation*}
then ${\rm Re} \{p(z)\}>0$ in $\Delta$.
\end{lemma}
This paper is organized as follows. In Section \ref{sec class} some properties of the classes $\mathcal{R}(\alpha,\beta)$ and $\mathcal{L}_\alpha(b)$ are studied. In Section \ref{sec conjeuture} we obtain the radius of univalence of 2-th section sum of $f\in\mathcal{R}(\alpha,\beta)$ and we conjecture that this radius is for every section sum of the function $f$ that belonging to the class $\mathcal{R}(\alpha,\beta)$.


\section{On the classes $\mathcal{R}(\alpha,\beta)$ and $\mathcal{L}_\alpha(b)$}
\label{sec class}
  At first, applying Hergoltz's Theorem \cite[p. 21]{Duren} we obtain the extreme points of $\mathcal{R}(\alpha,\beta)$ as follows:
\begin{equation}\label{fx}
  f_x(z)=z+4(1-\beta)\sum_{n=2}^{\infty}\frac{x^{n-1}}{n[n+1+(n-1)e^{i\alpha}]}z^n
  \quad(|x|=1).
\end{equation}
Since the coefficient bounds are maximized at an extreme point,
as an application of \eqref{fx}, we have
\begin{equation*}
  |a_n|\leq \frac{4(1-\beta)}{n|n+1+(n-1)e^{i\alpha}|}=\frac{4(1-\beta)}
  {n\sqrt{2[n^2+1+(n^2-1)\cos \alpha]}}\quad(n\geq2),
\end{equation*}
where $0\leq \beta<1$ and $-\pi<\alpha\leq \pi$.
Equality occurs for $f_x(z)$ defined by \eqref{fx}.

To prove the first result of this section, i.e. Theorem \ref{t3}, also Theorem \ref{th. f bar z} and Theorem \ref{t4}, we employ the same technique as in \cite[Theorem 2.1]{kessiberian}.
\begin{theorem}\label{t3}
  Let $\beta\in[0,1)$ and $\alpha\in(-\pi,\pi]$. If $f\in\mathcal{A}$ belongs to the class $\mathcal{R}(\alpha,\beta)$, then
  \begin{equation*}
    {\rm Re}\{f'(z)\}>\beta\quad(0\leq \beta<1).
  \end{equation*}
  This means that $\mathcal{R}(\alpha,\beta)\subset \mathcal{C}(\beta)$.
\end{theorem}
\begin{proof}
  Let $f'(z)\neq 0$ for $z\neq 0$ and $p(z)$ be defined by
\begin{equation*}
 p(z)=\frac{1}{1-\beta}\left(f'(z)-\beta\right)\quad(0\leq \beta<1).
\end{equation*}
Then $p(z)$ is analytic in $\Delta$, $p(0)=1$ and
\begin{equation*}
  f'(z)+\frac{1+e^{i \alpha}}{2}zf''(z)=(1-\beta)[p(z)+(1+e^{i \alpha})zp'(z)/2]+\beta=\phi(p(z),zp'(z);z),
\end{equation*}
  where $\phi(r,s;z):=(1-\beta)[r+(1+e^{i \alpha})s/2]$.
  Since $f\in\mathcal{R}(\alpha,\beta)$, we define the set $\Omega_\beta$ as follows:
  \begin{equation}\label{Omega beta}
    \{\phi(p(z),zp'(z);z):z\in \Delta\}\subset \{w: {\rm Re}\,\{w\} >\beta\}=:\Omega_\beta.
  \end{equation}
  For all real $\rho$ and $\sigma$, that $\sigma\leq -(1+\rho^2)/2$, we get
  \begin{align*}
    {\rm Re} \{\phi(i\rho, \sigma;z)\}& ={\rm Re} \left\{(1-\beta)[i\rho+(1+e^{i \alpha})\sigma/2]\right\} \\
    &= (1-\beta)(1+\cos \alpha)\sigma/2+\beta\\
    &\leq \beta-\frac{(1-\beta)}{4}(1+\cos \alpha)(1+\rho^2)\\
    &\leq \beta.
  \end{align*}
This shows that ${\rm Re}\{\phi(p(z),zp'(z);z)\}\not \in \Omega_\beta$. Thus by Lemma \ref{MIMO}, we get ${\rm Re} \{p(z)\}>0$ or ${\rm Re} \{f'(z)\}>\beta$. This means that $f\in \mathcal{C}(\beta)$ and concluding the proof.
\end{proof}
Taking $\beta=0$ in the above Theorem \ref{t3}, we get.
\begin{corollary}\label{co1}
  If $f\in \mathcal{L}_\alpha$, then ${\rm Re}\{f'(z)\}>0$ $(z\in \Delta)$ and thus $f$ is univalent.
\end{corollary}
\begin{remark}
  Since ${\rm Re}\{(1+e^{i\alpha})/2\}=(1+\cos \alpha)/2\geq 0$ where $\alpha\in(-\pi,\pi]$, thus the above Theorem \ref{t3} is a generalization of the results that earlier were obtained by Chichra \cite{Chichra} and Lewandowski {\it et al.} \cite{lewa et al}.
\end{remark}
The problem of finding a lower bound for ${\rm Re}\left\{f(z)/z\right\}$ is called
Marx-Strohh\"{a}cker problem. Marx and Strohh\"{a}cker (\cite{Marx, Str}) proved that if $f\in\mathcal{K}$, then ${\rm Re}\left\{f(z)/z\right\}>1/2$. In the sequel we consider this problem for the class $\mathcal{R}(\alpha,\beta)$.

\begin{theorem}\label{th. f bar z}
   Let $\beta\in[0,1)$ and $\alpha\in(-\pi,\pi]$. If $f\in\mathcal{A}$ belongs to the class $\mathcal{R}(\alpha,\beta)$, then we have
   \begin{equation*}
     {\rm Re}\left\{\frac{f(z)}{z}\right\}>\beta\quad(0\leq \beta<1).
   \end{equation*}
\end{theorem}
\begin{proof}
  Let the function $f\in\mathcal{A}$ belongs to the class $\mathcal{R}(\alpha,\beta)$ where $\beta\in[0,1)$ and $\alpha\in(-\pi,\pi]$. Define the function $\textsl{p}$ as
  \begin{equation}\label{p2}
    \textsl{p}(z):=\frac{1}{1-\beta}\left(\frac{f(z)}{z}-\beta\right).
  \end{equation}
  Since $f\in\mathcal{A}$, easily seen that $\textsl{p}$ is analytic in $\Delta$ and $\textsl{p}(0)=1$. The equation \eqref{p2}, with a simple calculation implies that
  \begin{equation}\label{f prime}
    f'(z)=\beta+(1-\beta)\textsl{p}(z)+(1-\beta)z\textsl{p}'(z)
  \end{equation}
  and
  \begin{equation}\label{f second}
    f''(z)=2(1-\beta)\textsl{p}'(z)+(1-\beta)z\textsl{p}''(z).
  \end{equation}
  Now, from \eqref{f prime} and multiplying \eqref{f second} by $ \frac{1+e^{i\alpha}}{2}z$, we get
  \begin{align*}
    &f'(z)+\frac{1+e^{i\alpha}}{2} zf''(z)\\
    &\quad=\beta+(1-\beta)\textsl{p}(z)+[(2+e^{i\alpha})(1-\beta)]z\textsl{p}'(z)+(1-\beta)
    \frac{1+e^{i\alpha}}{2} z^2\textsl{p}''(z)\\
    &\quad=\psi(\textsl{p}(z), z\textsl{p}'(z), z^2\textsl{p}''(z);z),
  \end{align*}
  where
  \begin{equation*}
    \psi(r,s,t;z)=\beta+(1-\beta)r+[(2+e^{i\alpha})(1-\beta)]s+(1-\beta)\frac{1+e^{i\alpha}}{2} t.
  \end{equation*}
  Since $f\in\mathcal{R}(\alpha,\beta)$ we consider the following inclusion relation
  \begin{equation*}
    \{\psi(\textsl{p}(z), z\textsl{p}'(z), z^2\textsl{p}''(z);z):z\in\Delta\}\subset \Omega_\beta,
  \end{equation*}
  where $\Omega_\beta$ is defined in \eqref{Omega beta}. Let $\rho,\sigma,\mu$ and $\nu$ be real numbers such that
  \begin{equation}\label{condition for rho sigma mu nu}
    \sigma\leq -\frac{1}{2}(1+\rho^2),~ \mu+\nu\leq0\quad{\rm and}\quad 2\mu+\nu\leq0.
  \end{equation}
  From \eqref{condition for rho sigma mu nu} and by \cite{MM Michigan} (see also \cite[Theorem 2.3b]{MM-book}), since
  \begin{equation*}
    \psi(i\rho, \sigma, \mu+i\nu ;z)=\beta+(1-\beta)i\rho+[(2+e^{i\alpha})(1-\beta)]\sigma+(1-\beta)\frac{1+e^{i\alpha}}{2}
    (\mu+i\nu),
  \end{equation*}
  we get
  \begin{align*}
    &\quad{\rm Re}\{\psi(i\rho, \sigma, \mu+i\nu ;z)\}\\
    &=\beta+(1-\beta)(2+\cos\alpha)\sigma+
    \frac{1-\beta}{2}[\mu(1+\cos\alpha)-\nu\sin\alpha]\\
    &\leq \beta-(1-\beta)(1+\rho^2)(1+\cos\alpha)+\frac{1-\beta}{2}[\mu(1+\cos\alpha)
    -\nu\sin\alpha]\\
    &=F(\alpha,\beta,\rho)+G(\alpha,\beta,\mu),
  \end{align*}
  where
  \begin{equation*}
    F(\alpha,\beta,\rho):=\beta-(1-\beta)(1+\rho^2)(1+\cos\alpha)
  \end{equation*}
  and
  \begin{equation*}
    G(\alpha,\beta,\mu):=\frac{1-\beta}{2}[\mu(1+\cos\alpha)-\nu\sin\alpha].
  \end{equation*}
  It is easy to see that $F(\alpha,\beta,\rho)\leq \beta$. Since $2\mu+\nu\leq 0$, we have $ G(\alpha,\beta,\mu)\leq 0$. Thus ${\rm Re}\{\psi(i\rho, \sigma, \mu+i\nu ;z)\}\leq \beta$ and this means that
  \begin{equation*}
  {\rm Re}\{\psi(\textsl{p}(z), z\textsl{p}'(z), z^2\textsl{p}''(z);z)\}\not\in\Omega_\beta.
  \end{equation*}
  Therefore we obtain ${\rm Re}\{\textsl{p}(z)\}>0$ where $\textsl{p}$ is given by \eqref{p2}, or equivalently
     \begin{equation*}
     {\rm Re}\left\{\frac{f(z)}{z}\right\}>\beta\quad(0\leq \beta<1).
   \end{equation*}
   This completes the proof.
\end{proof}
If we put $\beta=0$ in the above Theorem \ref{th. f bar z}, we get.
\begin{corollary}\label{co1}
  If $f\in \mathcal{L}_\alpha$, then ${\rm Re}\{f(z)/z\}>0$ in the open unit disc $\Delta$.
\end{corollary}
We shall require the following lemma in order to prove of the next result.
\begin{lemma}\label{Im phi lemma}
 Let $\phi_b(z)$ be defined by \eqref{phi} for $b>1/2$. Then $\phi_b(\Delta)=\Omega_b$ where
 \begin{equation*}
   \Omega_b:=\{w\in \mathbb{C}: 0< {\rm Re} \{w\}<2b\}.
 \end{equation*}
\end{lemma}
\begin{proof}
If $b=1$, then we have $0<{\rm Re}\{\phi_b(z)\}={\rm Re}\{1+z\}<2$. For $b>1/2$ and $b\neq 1$,
the function $\phi_b(z)$ does not have any poles in $\overline{\Delta}$ and is analytic in $\Delta$. Thus looking for the $\min\{{\rm Re}\{\phi_b(z)\}:~ |z|<1\}$ it is sufficient to consider it on the boundary $\partial \phi_b(\Delta)=\{\phi_b(e^{i\varphi}):\varphi \in [0,2\pi]\}$.
  A simple calculation gives us
  \begin{equation*}\label{ReFpsok}
    {\rm Re}\left\{\phi_b(e^{i\varphi})\right\}
    =\frac{(1/b)(1+\cos\varphi)}{1+2(1/b-1)\cos\varphi+(1/b-1)^2}
    \quad(\varphi\in[0,2\varphi]).
  \end{equation*}
  So we can see that ${\rm Re}\left\{F_{\alpha}(z)\right\}$ is well defined also for $\varphi=0$ and $\varphi=2\pi$. Define
  \begin{equation*}\label{h(x)}
    h(x)=\frac{(1/b)(1+x)}{1+2(1/b-1)x+(1/b-1)^2}\quad (-1\leq x\leq 1).
  \end{equation*}
  Thus for $b>1/2$ and $b\neq 1$, we have $h'(x)>0$. Therefore, we get
  \begin{equation*}
    0=h(-1)\leq h(x)\leq h(1)=2b.
  \end{equation*}
  This completes the proof.
\end{proof}

\begin{theorem}\label{t4}
  Let $f\in \mathcal{A}$ be a member of the class $\mathcal{L}_\alpha(b)$ where $b>1/2$ and $\alpha\in(-\pi,\pi]$. Then
  \begin{equation*}
    0< {\rm Re}\{f'(z)\}<2b \quad (z\in \Delta).
  \end{equation*}
\end{theorem}
\begin{proof}
Let us $f\in\mathcal{L}_\alpha(b)$. Then by Lemma \ref{lemma spelina}, Lemma \ref{Im phi lemma} and by definition of the subordination principle we have
\begin{equation}\label{eq1t4}
   0<{\rm Re}\left\{ f'(z)+\frac{1+e^{i\alpha}}{2}zf''(z)\right\}<2b\quad (z\in \Delta, b>1/2, -\pi<\alpha\leq \pi).
\end{equation}
First, we assume that
\begin{equation*}
  {\rm Re}\left\{ f'(z)+\frac{1+e^{i\alpha}}{2}zf''(z)\right\}>0.
\end{equation*}
Then by Corollary \ref{co1} we have ${\rm Re}\{f'(z)\}>0$.
Now we let
\begin{equation*}
{\rm Re}\left\{ f'(z)+\frac{1+e^{i\alpha}}{2}zf''(z)\right\}<2b.
\end{equation*}
Put $\xi=2b$ and so $\xi>1$. Let $f'(z)\neq 0$ for $z\neq 0$. Consider
\begin{equation*}
  q(z)=\frac{1}{1-\xi}\left(f'(z)-\xi\right)\quad (\xi>1, z\in \Delta).
\end{equation*}
Then $q(z)$ is analytic in $\Delta$ and $q(0)=1$. A simple check gives us
\begin{equation*}
  f'(z)+\frac{1+e^{i \alpha}}{2}zf''(z)=(1-\xi)[q(z)+(1+e^{i \alpha})zq'(z)/2]+\xi=\eta(q(z),zq'(z);z),
\end{equation*}
where $\eta(x,y;z)=(1-\xi)[x+(1+e^{i \alpha})y/2]+\xi$. Now we define
  \begin{equation*}
    \{\eta(q(z),zq'(z);z):z\in \Delta\}\subset \{w: {\rm Re}\{w\} <\xi\}=:\Omega_\xi.
  \end{equation*}
Again with a simple calculation we deduce that
  \begin{align*}
    {\rm Re} \{\eta(i\rho, \sigma;z)\}& ={\rm Re} \left\{(1-\xi)[i\rho+(1+e^{i \alpha})\sigma/2]\right\} \\
    &= (1-\xi)(1+\cos \alpha)\sigma/2+\xi\\
    &\geq \frac{(\xi-1)}{4}(1+\cos \alpha)(1+\rho^2)+\xi\quad (-\sigma\geq (1+\rho^2)/2)\\
    &\geq \xi.
  \end{align*}
This shows that ${\rm Re} \{\eta(i\rho, \sigma;z)\}\not \in \Omega_\xi$ and therefore ${\rm Re}\{q(z)\}>0$, or equivalently ${\rm Re} \{f'(z)\}<\xi$. This is the end of proof.
\end{proof}

\begin{theorem}\label{t5}
  Assume that $b>1/2$, $\alpha\in(-\pi,\pi]$ and $f\in \mathcal{L}_\alpha(b)$. Then for each $|z|=r<1$ we have
  \begin{equation}\label{1t5}
    1-\frac{(2b-1)r}{b+(b-1)r}\leq {\rm Re}\left\{ f'(z)+\frac{1+e^{i\alpha}}{2}zf''(z)\right\} \leq 1+\frac{(2b-1)r}{b+(b-1)r}.
  \end{equation}
\end{theorem}
\begin{proof}
  Let $f\in \mathcal{L}_\alpha(b)$. Then from the definition of subordination and by Lemma \ref{lemma spelina}, there exists a $\omega\in \mathcal{B}$ such that
  \begin{equation}\label{sub t31}
    f'(z)+\frac{1+e^{i\alpha}}{2}zf''(z)=\frac{1+\omega(z)}{1+(\frac{1}{b}-1)
    \omega(z)}\quad (z\in \Delta).
  \end{equation}
  We define
  \begin{equation*}
    W(z)=\frac{1+\omega(z)}{1+\left(\frac{1}{b}-1\right)\omega(z)},
  \end{equation*}
  which readily yields
  \begin{equation*}
    W(z)-1=\frac{(2-\frac{1}{b})\omega(z)}{1+\left(\frac{1}{b}-1\right)\omega(z)}.
  \end{equation*}
  For $|z| = r < 1$, using the known fact that (see \cite{Duren}) $|\omega(z)|\leq |z|$ we find that
  \begin{equation}\label{|w-1|}
    |W(z)-1|\leq \frac{(2b-1)r}{b+(b-1)r}.
  \end{equation}
  Hence $W(z)$ maps the disk $|z| < r < 1$ onto the disc which the center $C=1$ and the
radius $\delta$ given by
\begin{equation*}
  \delta=\frac{(2b-1)r}{b+(b-1)r}.
\end{equation*}
  Therefore,
  \begin{equation*}
    1-\frac{(2b-1)r}{b+(b-1)r}\leq |W(z)|\leq 1+\frac{(2b-1)r}{b+(b-1)r}.
  \end{equation*}
  Now, the assertion follows from \eqref{sub t31} and this fact that ${\rm Re}\{z\}\leq |z|$.
\end{proof}
\begin{remark}
  We obtained two lower and upper bounds for
  \begin{equation*}
    {\rm Re}\left\{ f'(z)+\frac{1+e^{i\alpha}}{2}zf''(z)\right\},
  \end{equation*}
  when $f\in \mathcal{L}_\alpha(b)$. From \eqref{eq1t4}, we have
  \begin{equation*}
    0<{\rm Re}\left\{ f'(z)+\frac{1+e^{i\alpha}}{2}zf''(z)\right\}<2b\quad (z\in \Delta, b>1/2, -\pi<\alpha\leq \pi),
  \end{equation*}
  while by \eqref{1t5}
    \begin{equation*}
    G(r):=1-\frac{(2b-1)r}{b+(b-1)r}\leq {\rm Re}\left\{ f'(z)+\frac{1+e^{i\alpha}}{2}zf''(z)\right\} \leq U(r):=1+\frac{(2b-1)r}{b+(b-1)r}.
  \end{equation*}
  It is easy to check that $U(r)<2b$ if $b\geq 1$ (or $b\rightarrow 1^+$) while $G(r)\geq 0$ for $1/2< b\leq1$ (or $b\rightarrow 1^-$).
\end{remark}

\begin{corollary}
  Let $f\in  \mathcal{L}_\alpha(1)$. Then we have
  \begin{equation*}
   1-r< {\rm Re}\left\{ f'(z)+\frac{1+e^{i\alpha}}{2}zf''(z)\right\}<1+r\quad(|z|=r<1).
  \end{equation*}
\end{corollary}

\begin{corollary}
  By a simple geometric observation and applying \eqref{sub t31} and \eqref{|w-1|}, we have
  \begin{equation*}
    \left|\arg \left\{f'(z)+\frac{1+e^{i\alpha}}{2}zf''(z)\right\}\right|<\arcsin \frac{(2b-1)r}{b+(b-1)r}\quad (|z|=r<1, b>1/2).
  \end{equation*}
\end{corollary}
\section{The radius of univalence of 2-th section sum of $f\in \mathcal{R}(\alpha,\beta)$}\label{sec conjeuture}
In this section, we obtain the radius of univalence of 2-th section sum of $f\in \mathcal{R}(\alpha,\beta)$. We recall that
the Taylor polynomial $s_k(z)=s_k(f)(z)$ of $f$ defined by
\begin{equation*}
 s_k(z)=s_k(f)(z)=z+a_2z^2+\cdots+ a_k z^k,
\end{equation*}
is called the $k-th\, section/partial$ sum of $f$. In \cite{szego}, proved that every section $s_k(z)$ of a $f\in \mathcal{S}$ is univalent in the disk $|z|<1/4$ and the number $1/4$ is best possible as the second partial sum of the Koebe function $k(z)=z/(1-z)^2$ shows. Next, we find the radius of univalence of the 2-th section sum of $f\in \mathcal{R}(\alpha,\beta)$.
\begin{theorem}
  The 2-th section sum of $f\in \mathcal{R}(\alpha,\beta)$ is univalent in the disc
\begin{equation*}
    |z|<\frac{\sqrt{10+6\cos \alpha}}{4(1-\beta)}\quad(-\pi<\alpha\leq \pi, 0\leq \beta<1).
\end{equation*}
The number $\frac{\sqrt{10+6\cos \alpha}}{4(1-\beta)}$ cannot be replaced by a greater one.
\end{theorem}
\begin{proof}
  Let $f(z)=z+\sum_{n=2}^{\infty}a_n z^n\in \mathcal{R}(\alpha,\beta)$ and $s_2(z)=z+a_2 z^2$ be its second section. By a simple calculation and since $|a_2|\leq \frac{2(1-\beta)}{\sqrt{10+6\cos \alpha}}$ we have
  \begin{equation*}
    {\rm Re}\{s_2'(z)\}={\rm Re}\{1+2a_2 z\}\geq 1-2|a_2||z|\geq 1-\frac{4(1-\beta)|z|}{\sqrt{10+6\cos \alpha}},
  \end{equation*}
  which is positive provided $|z|<\frac{\sqrt{10+6\cos \alpha}}{4(1-\beta)}$. Therefore $s_2(z)$ is close-to-convex (univalent) in the disk $|z|<\frac{\sqrt{10+6\cos \alpha}}{4(1-\beta)}$. To show that this bound is sharp, we consider the function $f_x$ defined by \eqref{fx}. The second partial sum $s_2(f_x)(z)$ of $f_x$ is $z+\frac{4(1-\beta)}{2(3+e^{i\alpha})}z^2$. Thus we get
  \begin{equation*}
    s_2'(z)=1+\frac{4(1-\beta)}{(3+e^{i\alpha})}z.
  \end{equation*}
  Hence ${\rm Re}\{s_2'(z)\}=0$ when $z=-\frac{(3+e^{i\alpha})}{4(1-\beta)}$. This completes the proof.
\end{proof}
We finish this paper with the following conjecture:
\newline{\bf Conjecture.} Every section of $f\in \mathcal{R}(\alpha,\beta)$ is univalent in the disc $|z|<\frac{\sqrt{10+6\cos \alpha}}{4(1-\beta)}$.

\end{document}